\documentclass[11pt]{amsart}
\usepackage[T1]{fontenc}
\usepackage[utf8]{inputenc}
\usepackage[usenames]{color}
\usepackage{graphicx}
\usepackage{float}
\usepackage{fancyhdr}
\usepackage{geometry}
\usepackage{subcaption} 
\input{xy}
\xyoption{all}
\xyoption{poly}
\usepackage[all]{xy}
\usepackage{mathrsfs}
\usepackage[]{color}
\usepackage{xcolor}
\usepackage{listings}[2013/08/05]

\usepackage{amsmath,amsthm,amsfonts,amssymb}
 
\theoremstyle{plain}
\newtheorem{theorem}{Theorem}[section]
\newtheorem*{theorem*}{Theorem}
\newtheorem{lemma}[theorem]{Lemma}
\newtheorem{corollary}[theorem]{Corollary}
\newtheorem{proposition}[theorem]{Proposition}

\newtheorem*{aproposition}{Proposition \ref{propReductionAschbacherConjecture}}
\newtheorem*{atheorem}{Theorem \ref{mainTheorem}}
\newtheorem*{aconjecture}{Aschbacher's Conjecture}

\theoremstyle{definition}
\newtheorem{example}[theorem]{Example}

\theoremstyle{remark}
\newtheorem{remark}[theorem]{Remark}

 \newcommand\ZZ{{\mathbb{Z}}}
 
 \def\A{{\mathcal A}}
 \def\B{{\mathcal B}}

 \def\K{{\mathcal K}}

 \def\N{{\mathcal N}}

 \def\S{{\mathcal S}}
 
 \def\U{{\mathcal U}}

 \newcommand{\tp}{{\mathscr{T}}}

 \def\Syl{{\text{Syl}}}
 \def\Max{{\text{Max}}}
 \def\Min{{\text{Min}}}
 \def\join{*}
 \def\Aut{{\text{Aut}}}
 \newcommand{\Hom}{\operatorname{Hom}}
 \def\Out{{\text{Out}}}
 
 \def\Inndiag{{\text{Inn diag}}}

\newcommand{\normal}{\trianglelefteq}

\newcommand\gen[1]{\left\langle#1\right\rangle}

      \makeatletter
      \def\@setcopyright{}
      \def\serieslogo@{}
      \makeatother
   
\begin{document}

\title [The fundamental group of the $p$-subgroup complex]{The fundamental group of the $p$-subgroup complex}
   \author{El\'ias Gabriel Minian}
   \author{Kevin Iv\'an Piterman}
   \address{Departamento  de Matem\'atica \\IMAS-CONICET\\
 FCEyN, Universidad de Buenos Aires. Buenos Aires, Argentina.}
\email{gminian@dm.uba.ar ; kpiterman@dm.uba.ar}

\thanks{Partially supported by grants UBACyT 20020160100081BA}

   \begin{abstract}
   
   We study the fundamental group of the $p$-subgroup complex of a finite group $G$. We show first that $\pi_1(\A_3(A_{10}))$ is not a free group (here $A_{10}$ is the alternating group on $10$ letters). This is the first concrete example in the literature of a $p$-subgroup complex with non-free fundamental group. We prove that, modulo a well-known conjecture of M. Aschbacher, $\pi_1(\A_p(G)) = \pi_1(\A_p(S_G)) * F$, where $F$ is a free group and $\pi_1(\A_p(S_G))$ is free if $S_G$ is not almost simple. Here $S_G = \Omega_1(G)/O_{p'}(\Omega_1(G))$. This result essentially reduces the study of the fundamental group of $p$-subgroup complexes to the almost simple case. We also exhibit various families of almost simple groups whose $p$-subgroup complexes have free fundamental group.
   \end{abstract}



\maketitle
\section{Introduction}

Let $G$ be a finite group and $p$ a prime number dividing its order. The posets $\S_p(G)$ and $\A_p(G)$ consist of, respectively, the non-trivial $p$-subgroups of $G$ and the non-trivial elementary abelian $p$-subgroups of $G$, ordered by inclusion. In \cite{Qui78} D. Quillen studied homotopical properties of the poset $\A_p(G)$ by means of its order complex (which is now called the Quillen complex of $G$ at $p$). Recall that the order complex of a finite poset $X$ is the simplicial complex $\K(X)$ with simplices the non-empty chains of $X$. Quillen showed that the inclusion $\A_p(G)\hookrightarrow \S_p(G)$ is a homotopy equivalence at the level of complexes and conjectured that $G$ has a non-trivial normal $p$-subgroup when $\K(\A_p(G))$ (or equivalently, $\K(\S_p(G))$) is contractible. Although Quillen's conjecture remains open, significant progress has been made. In the early nineties Aschbacher and Smith obtained the most relevant partial confirmation of the conjecture so far \cite{AS93}.

 The homotopy type of the $p$-subgroup complex is, in general, not known. Quillen showed that $\K(\A_p(G))$ has the homotopy type of a bouquet of spheres for a group $G$ of Lie type in characteristic $p$ or for $G = LA$, where $L$ is a solvable $p'$-group and $A$ is an elementary abelian $p$-group acting on $L$ (see \cite{Qui78}).  In \cite{PW00} Pulkus and Welker obtained a wedge decomposition for $\A_p(G)$ when $G$ has a solvable normal $p'$-subgroup. There is a question in \cite{PW00}, attributed to Th\'evenaz, whether the $p$-subgroup complex always has the homotopy type of a bouquet of spheres (of possibly different dimensions). In 2004 Shareshian gave the first example of a group whose $p$-subgroup complex is not homotopy equivalent to a bouquet of spheres. Concretely, he showed that there is torsion in the second homology group of $\A_3(S_{13})$ (see \cite{Sha04}). Here, $S_n$ denotes the symmetric group on $n$ letters.

The fundamental group of the $p$-subgroup complex was first studied by M. Aschbacher. He provided some algebraic conditions for $\A_p(G)$ to be simply connected, modulo a well-known conjecture for which there is considerable evidence (see \cite[Theorems 1 \& 2]{Asc93}). Later, K. Das studied the simple connectivity of the $p$-subgroups complexes of some groups $G$ of Lie type (see \cite{Das95,Das98,Das00}). In \cite{Kso03,Kso04}, Ksontini investigated the fundamental group of $\A_p(S_n)$ and proved that it is free except possibly for $n = 3p$ or $n=3p+1$. In \cite{Sha04}, Shareshian extended Ksontini's results and showed that the fundamental group of $\A_p(S_n)$ is also free for $n = 3p$.  All these results could suggest that the fundamental group of $\A_p(G)$ is always free. We will show that this holds for solvable groups (see Corollary \ref{coroSolvable} below) and, modulo Aschbacher's conjecture, for $p$-solvable groups (see Corollary \ref{coroPsolvable} below). In fact, there are only few known examples of $p$-subgroup complexes which are not homotopy equivalent to a bouquet of spheres, and Shareshian's counterexample $\A_3(S_{13})$ fails in the second homology group but it does have free fundamental group.  Surprisingly we found that the fundamental group of $\A_3(A_{10})$ is not free (here $A_{10}$ is the alternating group on $10$ letters). Its fundamental group is isomorphic to a free product of the free group on $25200$ generators and a non-free group whose abelianization is $\ZZ^{42}$. This is the smallest group $G$ with $\pi_1(\A_p(G))$ non-free for some $p$. Note that the integral homology of $\A_3(A_{10})$ ($=\A_3(S_{10})$) is free abelian (cf. \cite[p.306]{Sha04}), so in this case the obstruction to being a bouquet of spheres relies on the fundamental group and not on the homology. 

We will show that $p$-subgroup complexes with non-free fundamental group are rather exceptional. The first of our main results asserts that, modulo Aschbacher's conjecture, the study of freeness of $\pi_1(\A_p(G))$ reduces to the almost simple case. Concretely, we prove the following.

\begin{atheorem}


Let $G$ be a finite group and $p$ a prime dividing $|G|$. Assume that Aschbacher's conjecture holds. Then there is an isomorphism $\pi_1(\A_p(G)) = \pi_1(\A_p(S_G)) * F$, where $F$ is a free group. Moreover, $\pi_1(\A_p(S_G))$ is a free group (and therefore $\pi_1(\A_p(G))$ is free) except possibly if $S_G$ is almost simple.
\end{atheorem}

 Here $S_G = \Omega_1(G)/O_{p'}(\Omega_1(G))$. Recall that, for a fixed prime $p$, $\Omega_1(G) = \gen{x\in G: x^p = 1}$, $O_p(G)$ is the largest normal $p$-subgroup of $G$ and $O_{p'}(G)$ is the largest normal $p'$-subgroup of $G$ (i.e. of order prime to $p$).  We always can assume that $\A_p(G)$ is connected. Note that, in that case, $\A_p(S_G)$ is also connected since the induced map $\A_p(G)\to \A_p(S_G)$ is surjective. 
 
 In fact, in Theorem \ref{mainTheorem} we only need Aschbacher's conjecture to hold for the $p'$-simple groups $L$ involved in $O_{p'}(\Omega_1(G))$.  Here, a group $H$ is involved in $G$ if $H \simeq K/N$ for some $N\normal K\leq G$. Moreover, we only need the conjecture to hold for $p$-rank $3$ (see Proposition \ref{propReductionAschbacherConjecture} below).  Recall that the $p$-rank of a group $H$ is $m_p(H) = \max\{r : A\in\A_p(H), |A| = p^r\}$.


We recall now Aschbacher's conjecture \cite{Asc93}.

\begin{aconjecture}
Let $G$ be a finite group such that $G = AF^*(G)$, where $A$ is an elementary abelian $p$-subgroup of rank $r\geq 3$ and $F^*(G)$ is the direct product of the $A$-conjugates of a simple component $L$ of $G$ of order prime to $p$. Then $\A_p(G)$ is simply connected.
\end{aconjecture}

Here $F^*(G)$ denotes the generalized Fitting subgroup of $G$ (see Section \ref{sectionthree} or \cite{AscFGT}).
Aschbacher proved the conjecture for all simple groups $L$ except for Lie type groups with Lie rank 1 and the sporadic groups which are not Mathieu groups (see \cite[Theorem 3]{Asc93}).

 Here below we list some immediate consequences of Theorem \ref{mainTheorem}.

If $G$ is $p$-solvable, $O_p(S_G)\neq 1$, so $S_G$ is not almost simple. In fact, $\A_p(S_G)$ is contractible by Proposition \ref{propOpGContractible} below. Let $O_{p',p}(G)$ be the unique normal subgroup of $G$ containing $O_{p'}(G)$ such that $O_{p',p}(G)/O_{p'}(G) = O_p(G/O_{p'}(G))$.

\begin{corollary}\label{coroPsolvable}
Assume that Aschbacher's conjecture holds. If $O_{p'}(\Omega_1(G)) < O_{p',p}(\Omega_1(G))$, then $\pi_1(\A_p(G))$ is free. In particular, this holds for $p$-solvable groups.
\end{corollary}

If we take an abelian simple group $L$ in the hypotheses of Aschbacher’s conjecture, then the conjecture holds for $L$ by \cite[Theorem 11.2]{Qui78}. Since there are no non-abelian simple groups involved in a solvable group G, Aschbacher's conjecture does not need to be assumed for solvable groups.

\begin{corollary}\label{coroSolvable}
If $G$ is solvable then $\pi_1(\A_p(G))$ is a free group.
\end{corollary}

By Feit-Thompson, if $p = 2$ then there are no non-abelian $p'$-simple group. In consequence, Aschbacher's conjecture does not need to be assumed and we get the following corollary.

\begin{corollary}
There is an isomorphism $\pi_1(\A_2(G)) \simeq \pi_1(\A_2(S_G)) * F$, where $F$ is a free group. Moreover, $\pi_1(\A_2(S_G))$ is a free group (and therefore $\pi_1(\A_2(G))$ is free) except possibly if $S_G$ is almost simple.
\end{corollary}

In Section \ref{sectionfour}, we use Pulkus-Welker's wedge decomposition \cite{PW00} to restrict Aschbacher's conjecture to the $p$-rank $3$ case.

\begin{aproposition}
If Aschbacher's conjecture holds for $p$-rank $3$, then it holds for any $p$-rank $r\geq 3$. Moreover, if the conjecture holds in $p$-rank $3$ for a $p'$-simple group $L$ then it holds in any $p$-rank $r\geq 3$ for $L$.
\end{aproposition}

In Section \ref{sectionsix} we study freeness for some particular cases of almost simple groups. We do not need to assume Aschbacher's conjecture for these cases. In the following theorem we collect the results of Section \ref{sectionsix}.  We use the notations of the finite simple groups of \cite{GLS98}. In general, by a simple group we will mean a non-abelian simple group. 

\begin{theorem}\label{theoremAlmostSimpleFree}
Suppose that $L \leq G \leq \Aut(L)$, with $L$ a simple group. Then $\pi_1(\A_p(G))$ is a free group in the following cases:
\begin{enumerate}
\item $m_p(G)\leq 2$,
\item $\A_p(L)$ is disconnected,
\item $\A_p(L)$ is simply connected,
\item $L$ is simple of Lie type in characteristic $p$ and $p\nmid (G:L)$ when $L$ has Lie rank $2$,
\item $p = 2$ and $L$ has abelian Sylow $2$-subgroups,
\item $p = 2$ and $L = A_n$ (the alternating group),
\item $L$ is a Mathieu group,
\item $L = J_1$ or $J_2$,
\item $p \geq 3$ and $L = J_3$, $Mc$, $O'N$.
\end{enumerate}
\end{theorem}

From our base example $A_{10}$ (with $p=3$) and Theorem \ref{mainTheorem}, one can easily construct  an infinite number of examples of finite groups $G$ with non-free $\pi_1(\A_3(G))$, by taking extensions of $3'$-groups $H$ whose $3'$-simple groups involved satisfy Aschbacher's conjecture, by $A_{10}$. However $A_{10}$ is the unique known example so far of a simple group with non-free fundamental group. We do not know whether $\pi_1(\A_p(A_{3p+1}))$ is non-free for $p\geq 5$. It would be interesting to find new examples of simple groups $G$ (other than the alternating groups) with $\pi_1(\A_p(G))$ non-free. Besides the works of Aschbacher, Das, Ksontini and Shareshian mentioned above, we refer the reader to S. Smith's book \cite[Section 9.3]{Smi11} for more details on the fundamental groups of Quillen complexes and applications to group theory, such as uniqueness proofs. Also a recent work of J. Grodal  \cite{Gro16} relates the fundamental group of the $p$-subgroup complexes to modular representation theory of finite groups via the exact sequence 
$$1\to \pi_1(\S_p(G))\to\pi_1(\tp_p(G))\to G\to 1$$ 
(when $\S_p(G)$ is connected). Here $\tp_p(G)$ denotes the transport category, whose objects are the non-trivial $p$-subgroups of a (fixed) Sylow $p$-subgroup $S\leq G$, with  $\Hom_{\tp_p(G)}(P,Q) = \{ g
  \in G| {} P^g \leq Q\}$. In \cite[Remark 2.2]{Gro16} it is shown that the geometric realization of $\tp_p(G)$ is homotopy equivalent to the Borel construction $EG\times_G \S_p(G)$, and the exact sequence follows from the fibration sequence $\S_p(G)\to EG\times_G \S_p(G) \to BG$. Recall that, by Brown's ampleness theorem, the mod-$p$ cohomology of $EG\times_G \S_p(G)$ is isomorphic to the mod-$p$ cohomology of $G$ (see \cite{Bro94,Smi11}). We hope that the results of this article can shed more light on the topology of these objects.

In this paper we study the posets of $p$-subgroups topologically by means of their order complexes. We will say, for instance, that $\A_p(G)$ is contractible if its order complex  $\K(\A_p(G))$ is. Also, the homology groups and the fundamental group of the posets are those of their associated complexes.   Note that this is not the convention that we adopted in our previous articles \cite{MP18,Pit19}. In those papers, we handled the posets of $p$-subgroups as finite topological spaces, with an intrinsic topology, where the notion of homotopy equivalence is strictly stronger than in the context of simplicial complexes. However in this article we adopt the more usual convention (as, for example, in [Qui78, AS93, Asc93, PW00,    Sha04, Smi11]).

{\bf Acknowledgements.} We would like to thank Volkmar Welker, John Shareshian and Jesper Grodal for  helpful comments.
 

\section{A non-free fundamental group}\label{sectiontwo}

The fundamental group of the Quillen complex was first investigated by Aschbacher, who analyzed simple connectivity \cite{Asc93}. K. Das studied simple connectivity of the $p$-subgroups complexes of  groups of Lie type (see \cite{Das95,Das98,Das00}). In \cite{Kso03,Kso04}, Ksontini investigated the fundamental group of the Quillen complex of symmetric groups. Below we recall Ksontini's results. These results will be used in Proposition \ref{propAnCase}.

\begin{theorem}[{\cite{Kso03, Kso04}}]\label{theoremKsontini}
Let $G = S_n$ and let $p$ be a prime.
\begin{enumerate}
\item If $p$ is odd, then $\A_p(S_n) = \A_p(A_n)$. In this case, $\A_p(S_n)$ is simply connected if and only if $3p+2\leq n < p^2$ or $n\geq p^2+p$. If $p^2 \leq n < p^2+p$, then $\pi_1(\A_p(S_n))$ is free unless $p = 3$ and $n = 10$. If $n < 3p$ then $m_p(S_n) \leq 2$ and $\pi_1(\A_p(S_n))$ is free.
\item If $p = 2$, then $\A_2(S_n)$ is simply connected if and only if $n = 4$ or $n\geq 7$. In other cases, $\pi_1(\A_2(S_n))$ is a free group by direct computation.
\end{enumerate}
\end{theorem}

 In \cite{Sha04} Shareshian extended Ksontini's results and showed that the fundamental group of $\A_p(S_n)$ is also free for $n = 3p$. 
 
\begin{theorem}[{\cite{Sha04}}]\label{theoremShareshian}
$\pi_1(\A_p(S_n))$ is free when $n = 3p$.
\end{theorem}

In \cite{Sha04} Shareshian gave the first example of a group whose $p$-subgroup complex is not homotopy equivalent to a bouquet of spheres: he showed that there is torsion in the second homology group of $\A_3(S_{13})$. However its fundamental group is free. Surprisingly we found that the fundamental group of $\A_3(A_{10})$ is not free. This is the first concrete known example of a Quillen complex with non-free fundamental group. In fact, $A_{10}$ is, so far, the unique known example of a simple group whose Quillen complex has non-free fundamental group. 

To compute $\pi_1(\A_3(A_{10}))$ we used the Bouc poset of non-trivial $p$-radical subgroups. Recall that $R\leq G$ is said to be $p$-radical if $R = O_p(N_G(R))$. Denote by $\B_p(G)$ the poset of non-trival $p$-radical subgroups of $G$. In \cite{Bou84, TW91} it is shown that $\B_p(G)\hookrightarrow \S_p(G)$ is a homotopy equivalence at the level of complexes. In particular, $\pi_1(\A_p(G)) = \pi_1(\S_p(G)) = \pi_1(\B_p(G))$. 

We calculated $\pi_1(\A_3(A_{10}))$ using GAP (see \cite{Gap}). We wrote two functions that can be used to compute the order complex of the Bouc poset $\B_p(G)$ of a given finite group $G$ at a given prime $p$, and its fundamental group. The codes are shown in the Appendix.

We found that $\pi_1(\A_3(A_{10}))$ is a free product of the free group on $25200$ generators and a non-free group on $42$ generators and $861$ relators whose abelianization is $\ZZ^{42}$. It does not have torsion elements but it has commuting relations. Note that the integral homology of $\A_3(A_{10})$ is free abelian (cf. \cite[p.306]{Sha04}). As a consequence of Theorem \ref{mainTheorem}, one can  construct  an infinite number of examples of finite groups $G$ with non-free $\pi_1(\A_3(G))$, by taking extensions of $3'$-groups whose $3'$-simple groups involved satisfy Aschbacher's conjecture, by $A_{10}$. As we mentioned in the introduction, it would be interesting to find other examples of simple groups with non-free fundamental group.

We were able to verify that $A_{10}$ is the smallest group with a Quillen complex with non-free $\pi_1$.  Note that, by Theorem \ref{mainTheorem}, we only need to verify freeness in almost simple groups (note also that Aschbacher's conjecture holds for groups of order less than the order of $A_{10}$). On the other hand, Theorem \ref{theoremAlmostSimpleFree} allowed us to discard many potential counterexamples. The remaining almost simple groups which are smaller than $A_{10}$ were checked by computer calculations.

\section{Notations and preliminary algebraic and topological results}\label{sectionthree}

We fix notations and recall some basic definitions. We refer the reader to Aschbacher's book \cite{AscFGT} for more details on finite group theory.

For a finite group $G$ and a fixed prime number $p$ divining its order $|G|$, denote by $\Omega_1(G) = \gen{x\in G:x^p=1}$. The $p$-rank of $G$, denoted by $m_p(G)$, is the maximal integer $r$ such that there exists an elementary abelian $p$-subgroup of $G$ of order $p^r$. Recall that $O_p(G)$ is the largest normal $p$-subgroup of $G$ and $O_{p'}(G)$ is the largest normal $p'$-subgroup of $G$.
If $T, H\leq G$, then $C_T(H)=\{x\in T: xy=yx\text{ for all }y\in H\}$ and $N_T(H) = \{x\in T: H^x = H\}$ are, respectively, the centralizer and the normalizer of $H$ in $T$. Here, $H^x = x^{-1}Hx$. 

  $F^*(G)$, $Z(G)$, $G'$, $\Phi(G)$ denote, respectively, the generalized Fitting subgroup, the center, the derived subgroup and the Frattini subgroup of $G$. By definition, $F^* (G) = F(G) E(G)$, where $F(G)$ is the Fitting subgroup of $G$ and $E(G)$ is the layer of $G$, i.e. the subgroup generated by the components of $G$ (the subnormal quasisimple subgroups). Recall that $C_G(F^*(G)) \leq F^*(G)$, $[F(G),E(G)] = 1$, and that $[L_1,L_2] = 1$ if $L_1,L_2$ are distinct components of $G$. By a simple group we will mean a non-abelian simple group. We write $L\leq G\leq\Aut(L)$, with $L$ a simple group, for an almost simple group $G$ with $F^*(G) = L$. 

Denote by $D_n$ the dihedral group of order $n$, $A_n$ the alternating group on $n$ letters, $S_n$ the symmetric group on $n$ letters and $C_n$ the cyclic group of order $n$. The symbol $G\rtimes T$ denotes a split extension of $G$ by $T$. We write $G\simeq H$ if $G$ and $H$ are isomorphic groups.

By the classification of the finite simple groups, a simple group is either an alternating group $A_n$, a group of Lie type or one of the $26$ sporadic group. Throughout this paper, we will follow the notation of \cite{GLS98} for the simple groups.

As we mentioned in the introduction, $\K(X)$ denotes the order complex of a finite poset $X$. We denote by $|\K(X)|$ the geometric realization of $\K(X)$. Any poset map $f:X\to Y$ (i.e. order preserving) induces a simplicial map $\K(f):\K(X)\to \K(Y)$ and thus, a continuous map $|\K(f)|$ between the geometric realizations. If $f,g:X\to Y$ are poset maps with $f\leq g$ (i.e. $f(x)\leq g(x)$ for all $x\in X$) or $f\geq g$, then $|\K(f)|$ and $|\K(g)|$ are homotopic (see \cite{Qui78}). We use the symbol $\simeq $ to denote a homotopy equivalence between topological spaces (or posets).

If $X$ is a poset and $a\in X$, set $X_{\geq a} = \{x\in X: x\geq a\}$. Analogously we define $X_{>a}$, $X_{\leq a}$, $X_{<a}$. Denote by $\Max(X)$ (resp. $\Min(X)$) the set of maximal (resp. minimal) elements of $X$. If $A,B\subseteq C$ are sets, then $B - A = \{b\in B: b\notin A\}$ denotes the complement of $A$ in $B$.

The following result follows immediately from \cite[Corollary 4.10]{BM12}. We include here an alternative proof.

\begin{proposition}\label{propCollapsingSCSubspace}
Let $X$ be a finite connected poset and let $Y\subseteq X$ be a subposet such that $X-Y$ is an anti-chain (i.e. $\forall \ x,x'\in X-Y$, $x$ and $x'$  are not comparable). If $Y$ is simply connected, then $\pi_1(X)$ is free.
\end{proposition}

\begin{proof}
Since the inclusion $|\K(Y)|\subseteq |\K(X)|$ is a cofibration and $|\K(Y)|$ is simply connected, by van Kampen theorem there is an isomorphism $\pi_1(|\K(X)|) \simeq \pi_1(|\K(X)| / |\K(Y)|)$ induced by the quotient map. Since  $X-Y$ is an anti-chain, the space $|\K(X)| / |\K(Y)|$ has the homotopy type of a wedge of suspensions. Therefore, $\pi_1(X) = \pi_1(|\K(X)|) \simeq \pi_1(|\K(X)| / |\K(Y)|)$ is a free group.
\end{proof}

Recall that a topological space $X$ is $n$-connected if $\pi_i(X)$ is trivial for $i\leq n$. By convention, $X$ is $(-1)$-connected if and only if it is non-empty. A map $f:X\to Y$ between topological spaces is an $n$-equivalence if $f_ *:\pi_i(X) \to \pi_i(Y)$ is an isomorphism for $i< n$ and an epimorphism for $i = n$. A finite CW-complex of dimension $n$ is said to be $n$-spherical if it is $(n-1)$-connected. Equivalently, it has the homotopy type of a bouquet of spheres of dimension $n$ (see \cite[Section 8]{Qui78}).

\begin{proposition}[{\cite[Theorem 2]{Bjo03}}, see also \cite{Bar11b}]\label{propNConnectedFibers} Let $f:X\to Y$ be a map between posets. Assume that $f^{-1}(Y_{\leq a})$ is $n$-connected for all $a\in Y$. Then $f$ is an $(n+1)$-equivalence.
\end{proposition}

Denote by $X^n$ the subposet of $X$ consisting of elements $a$ such that the complex $\K(X_{\leq a})$ has dimension at most $n$. Note that $\S_p(G)^n = \{P\in \S_p(G) : |P|\leq p^{n+1}\}$ and $\A_p(G)^n = \A_p(G)\cap\S_p(G)^n$. Recall that $\A_p(G)\hookrightarrow \S_p(G)$ is a homotopy equivalence by \cite[Proposition 2.1]{Qui78}.

\begin{proposition}\label{propHeightReduction}
The inclusions $\A_p(G)^n\hookrightarrow\A_p(G)^{n+1}$ and $\S_p(G)^n\hookrightarrow\S_p(G)^{n+1}$ are $n$-equivalences. In particular, the inclusions $\A_p(G)^2\hookrightarrow\A_p(G)$ and  $\S_p(G)^2\hookrightarrow\S_p(G)$ induce isomorphisms between the fundamental groups.
\end{proposition}

\begin{proof}
We show that the inclusion $i:\S_p(G)^{n}\hookrightarrow\S_p(G)^{n+1}$ is an $n$-equivalence by using the previous proposition. Let $P\in \S_p(G)^{n+1}$. Note that $i^{-1}(\S_p(G)^{n+1}_{\leq P}) \subseteq \S_p(P)$. If $|P|\leq p^{n+1}$, then $P\in \S_p(G)^n$ and $i^{-1}(\S_p(G)^{n+1}_{\leq P}) = \S_p(P)$ is contractible (see Proposition \ref{propOpGContractible} below). In particular, it is $(n-1)$-connected. Suppose $|P| = p^{n+2}$. If $P$ is elementary abelian, then $i^{-1}(\S_p(G)^{n+1}_{\leq P}) = \A_p(P)-P$ is the poset of proper subspaces of $P$, which is a wedge of spheres of dimension $(n-1)$ by the classical Solomon-Tits result. If $P$ is not elementary abelian, $i^{-1}(\S_p(G)^{n+1}_{\leq P}) = \S_p(P)-P$,  and $1\neq \Phi(P) < P$, where $\Phi(P)$ is the Frattini subgroup of $P$. Then, $Q\leq Q\Phi(P)\geq \Phi(P)$ induces a homotopy between the identity map and the constant map inside $\S_p(P)-P$, and therefore $\S_p(P)-P$ is contractible. This shows that  $i:\S_p(G)^{n}\hookrightarrow\S_p(G)^{n+1}$ is an
$n$-equivalence. A similar proof works for $\A_p(G)$.
\end{proof}

\begin{remark}
By Proposition \ref{propHeightReduction}, in order to study the fundamental group of the Quillen complex, we only need to deal with the subposet $\A_p(G)^2$. In fact we will not use the posets $\S_p(G)$ or $\S_p(G)^2$. Note that we could have deduced the isomorphism $i_*:\pi_1(\A_p(G)^2)\to\pi_1(\A_p(G))$ without need of Propositions \ref{propNConnectedFibers} and \ref{propHeightReduction}: it follows from van Kampen theorem and the fact that for any $P\in \A_p(G)-\A_p(G)^2$, $\A_p(G)_{<P}$ is a wedge of spheres of dimension greater than or equal to $2$. We decided to include Proposition \ref{propHeightReduction} for future references. 
\end{remark}





The following result was proved by Quillen \cite[Proposition 2.4]{Qui78}.

\begin{proposition}\label{propOpGContractible}
If $O_p(G)\neq 1$ then $\A_p(G)$ and $\S_p(G)$ are contractible.
\end{proposition}

The converse of this proposition is Quillen's conjecture \cite[Conjecture 2.9]{Qui78}.

\begin{remark}\label{remarkInflationSubsapce}
For any subgroup $H\leq G$, consider the subposet $\mathcal{N}(H) = \{T\in \A_p(G) : T\cap H \neq 1\}\subseteq \A_p(G)$. Note that the inclusion $\A_p(H)\subseteq \mathcal{N}(H)$ is a strong deformation retract via $T\in \mathcal{N}(H)\mapsto T\cap H \in \A_p(H)$.

\end{remark}

\begin{lemma}\label{lemmaLinkRetract}
Let $H\leq G$ and let $T\in \A_p(G) - \A_p(H)$. Then $i:\A_p(C_H(T)) \to \N(H)\cap \A_p(G)_{>T}$ defined by $i(A) = AT$ is a strong deformation retract.
\end{lemma}

\begin{proof}
The result is clear if $\A_p(C_H(T))$ is empty. If it is not empty, let $r:\N(H)\cap \A_p(G)_{>T}\to \A_p(C_H(T))$ be the map $r(A) = A\cap H$. Then $ri(A) = A$ by modular law, and $ir(A)\leq A$.
\end{proof}







We will use the following lemma of \cite{Asc93}.

\begin{lemma}[{\cite[(6.9)]{Asc93}}]\label{lemmaConnectedLinks}
Let $N\normal G$ and suppose $\A_p(N)$ is simply connected. If $\A_p(C_N(T))$ is connected for each subgroup $T\leq G$ of order $p$, then $\A_p(G)$ is simply connected.
\end{lemma}

If $X$ and $Y$ are posets, their join is the poset $X * Y$ whose underlying set is the disjoint union $X\coprod Y$ and with the following order relation $\leq$. We keep the given order in $X$ and $Y$, and set $x\leq y$ for all $x\in X$ and $y\in Y$. It is easy to see that $\K(X*Y) = \K(X) * \K(Y)$ and therefore  $|\K(X * Y)|= |\K(X)| * |\K(Y)|$ (see \cite[Proposition 1.9]{Qui78}).

\begin{proposition}[{\cite[Lemma 6.2.4]{Bar11}}]\label{propFreePi1Join} 
If $X$ and $Y$ are finite non-empty posets, then $\pi_1(X*Y)$ is a free group of rank $(|\pi_0(X)|-1).(|\pi_0(Y)|-1)$.
\end{proposition}

\begin{proposition}[{\cite[Proposition 2.6]{Qui78}}]\label{propQuillenJoin}
If $G = G_1\times G_2$, then $\A_p(G) \simeq \A_p(G_1) * \A_p(G_2)$. In particular, if $p\mid \gcd{(|G_1|,|G_2|)}$, $\A_p(G)$ is connected and $\pi_1(\A_p(G))$ is free. Moreover, $\A_p(G)$ is simply connected if and only if $\A_p(G_i)$ is connected for some $i=1,2$.
\end{proposition}

Recall that $\A_p(G)$ is disconnected if and only if $G$ has a strongly $p$-embedded subgroup (see \cite[Proposition 5.2]{Qui78}). The groups with this property are classified and we will use this classification later.

\begin{theorem}[{\cite[(6.1)]{Asc93}}]\label{disconnectedCasesTheorem}
The poset $\A_p(G)$ is disconnected if and only if either $O_p(G) = 1$ and $m_p(G) = 1$, or $\Omega_1(G) / O_{p'}(\Omega_1(G))$ is one of the following groups:
\begin{enumerate}
\item Simple of Lie type of Lie rank $1$ and characteristic $p$,
\item $A_{2p}$ with $p \geq 5$,
\item ${}^2G_2(3)$, $L_3(4)$ or $M_{11}$ with $p = 3$,
\item $\Aut(Sz(32))$, ${}^2F_4(2)'$, $Mc$, or $Fi_{22}$ with $p = 5$,
\item $J_4$ with $p = 11$.
\end{enumerate}
\end{theorem}

\begin{remark}\label{remarkLieRank1}
The simple groups of Lie type and Lie rank $1$ are the groups $L_2(q)$, $U_3(q)$, $Sz(q)$ and ${}^2G_2(q)$. In characteristic $2$, these are $L_2(2^n)$, $U_3(2^n)$ and $Sz(2^n)$ and they are the unique simple groups with a strongly $2$-embedded subgroup. There are no simple groups of $2$-rank $1$ (see \cite{GLS98}).
\end{remark}


\begin{remark}\label{remarkConnectedComponents}
Suppose $\A_p(G)$ is disconnected and let $C$ be a connected component. Let $M\leq G$ be the stabilizer of $C$ under the conjugation of $G$ on the connected components of $\A_p(G)$. It can be shown that $C = \A_p(M)$ and that $M$ is a strongly $p$-embedded subgroup of $G$ (see for example \cite[Section 46]{AscFGT} and \cite[Section 5]{Qui78}). Moreover, since $G$ permutes transitively the connected components of $\A_p(G)$, they have isomorphic fundamental groups. This allows us to define $\pi_1(\A_p(G))$ as  $\pi_1(\A_p(M))$ (for any connected component $C$). Therefore the study of the fundamental group of the Quillen complexes can be restricted to the connected case.
\end{remark}

\section{Reduction to $O_{p'}(G) = 1 = O_p(G)$, $G = \Omega_1(G)$} \label{sectionfour}

In this section, we reduce the study of the fundamental group of $\A_p(G)$ to the case $O_{p'}(G) = 1 =O_p(G)$ and $G = \Omega_1(G)$. We assume that $\A_p(G)$ is connected.

The reduction $G = \Omega_1(G)$ is clear since $\A_p(G) = \A_p(\Omega_1(G))$. If $O_p(G) \neq 1$, then $\A_p(G)$ is contractible by Proposition \ref{propOpGContractible} and in particular simply connected. Therefore, we may assume $O_p(G) = 1$.

The reduction $O_{p'}(G) = 1$ relies on the wedge lemma of homotopy colimits. We will use Pulkus-Welker's result \cite[Theorem 1.1]{PW00} but for $\A_p(G)^2$ instead of $\A_p(G)$. Recall that $\pi_1(\A_p(G)) = \pi_1(\A_p(G)^2)$ by Proposition \ref{propHeightReduction}. 

Note that $m_p(G) = m_p(G/O_{p'}(G))$ 
and that $\A_p(G/O_{p'}(G))$ is connected when $\A_p(G)$ is connected since the induced map $\A_p(G) \to \A_p(G/O_{p'}(G))$ is surjective. The following lemma is a slight variation of Pulkus-Welker's result \cite[Theorem 1.1]{PW00}. Note that we have replaced the hypothesis of solvability of the normal $p'$-subgroup $N\leq G$  in \cite[Theorem 1.1]{PW00} by simple connectivity of $\A_p(AN)$ for $A$ of $p$-rank $3$.

\begin{lemma}\label{lemmaPW}
Let $N$ be a normal $p'$-subgroup of $G$ such that $\A_p(AN)$ is simply connected for each elementary abelian $p$-subgroup $A\leq G$ of $p$-rank $3$. Then
\[\A_p(G)^2\simeq \A_p(G/N)^2\bigvee_{\overline{B}\in \A_p(G/N)^2}\A_p(BN) * \A_p(G/N)^2_{>\overline{B}}\]
In particular, for a suitable base point,
\[\pi_1(\A_p(G)) = \pi_1(\A_p(G/N)) *_{\overline{B}\in \A_p(G/N)^2} \pi_1(\A_p(BN) * \A_p(G/N)^2_{>\overline{B}}).\]
\end{lemma}

\begin{proof}
We essentially follow the proof of Pulkus-Welker \cite[Theorem 1.1]{PW00}. Let $N\leq G$ be a normal $p'$-subgroup of $G$. Write $\overline{G} = G/N$ and let $f:\A_p(G)^2\to \A_p(\overline{G})^2$ be the map induced by taking quotients. Note that it is well defined and surjective. We will use \cite[Corollary 2.4]{PW00}. For this, we have to verify that the inclusions $f^{-1}(\A_p(\overline{G})^2_{<\overline{B}})\hookrightarrow f^{-1}(\A_p(\overline{G})^2_{\leq \overline{B}})$ are homotopic to constant maps. Note that $f^{-1}(\A_p(\overline{G})^2_{<\overline{B}}) = \A_p(BN) - \Max(\A_p(BN))$ and $f^{-1}(\A_p(\overline{G})^2_{\leq \overline{B}}) =  \A_p(BN)$.

By hypothesis and Remark \ref{remarkConnected} below we deduce that $\A_p(BN)-\Max(\A_p(BN))$ and $\A_p(BN)$ are spherical of the corresponding dimension for each $B\leq G$ of $p$-rank at most $3$. For instance, if $B$ has $p$-rank $3$, then $\A_p(BN)-\Max(\A_p(BN))$ is $0$-spherical and $\A_p(BN)$ is $1$-spherical.

The result now follows from the fact that the inclusion of a sphere of dimension $n$ into a sphere of dimension $m>n$ is homotopic to a constant map, and $\A_p(BN)-\Max(\A_p(BN))$ and $\A_p(BN)$ are spherical.
\end{proof}

\begin{remark}\label{remarkConnected}
Let $A$ be an elementary abelian $p$-group of $p$-rank at least $2$ acting on a $p'$-group $N$. Then $\A_p(AN)$ is connected. To show this, assume otherwise and take a minimal counterexample $AN$. Thus, $1 = O_p(AN)$ and $AN =\Omega_1(AN)$ by minimality. Therefore, $N=O_{p'}(AN) = O_{p'}(\Omega_1(AN))$ and $A = AN/N = \Omega_1(AN) / O_{p'}(\Omega_1(AN))$ is one of the groups in the list of Theorem \ref{disconnectedCasesTheorem}. But none of the groups in that list is elementary abelian of $p$-rank at least $2$. Therefore, $\A_p(AN)$ is connected.
\end{remark}

\begin{lemma}\label{lemmaJoinApgLinks}
Let $N$ be a normal $p'$-subgroup of $G$ and let $A\in \A_p(G)$ of $p$-rank at most $3$. Assume Aschbacher's conjecture for $p$-rank $3$. Then, the fundamental group of $\A_p(AN)*\A_p(G/N)^2_{>\overline{A}}$ is free if $|A|=p$ or $p^2$ and trivial if $|A|=p^3$.
\end{lemma}

\begin{proof}
We can suppose $N\neq 1$. We examine the possible ranks of $A$. If $|A|=p$, $\A_p(AN)$ is a disjoint union of points while $\A_p(G/N)^2_{>A}$ is a non-empty graph. Thus, their join is homotopic to a wedge of $2$-spheres and $1$-spheres.

If $|A|=p^2$, $\A_p(AN)$ is a connected non-empty graph by the above remark. Note that $\A_p(G/N)^2_{>A}$ may be either empty, if $A$ is maximal, or discrete. Thus, their join is homotopic to a wedge of $2$-spheres or $1$-spheres.

It remains the case $|A|=p^3$. Here, $\A_p(G/N)^2_{>A}$ is empty. Suppose $\A_p(AN)$ is not simply connected and take $N$ a minimal counterexample. We will show that the group $AN$ satisfies the hypotheses of Aschbacher's conjecture (and this leads to a contradiction). We may assume that $C_A(N) = O_p(AN) = 1$, so $A$ acts faithfully on $N$. Suppose that $N$ has a non-trivial proper normal subgroup $H$ which is also $A$-invariant. We apply Lemma \ref{lemmaPW} with $G = AN$ and $H$ as the normal $p'$-subgroup since $\A_p(BH)$ is simply connected for $B\leq AN$ of $p$-rank $3$ by minimality of $N$. Therefore,
\begin{equation}\label{equationWedge}
\A_p(AN) \simeq \A_p(A(N/H)) \bigvee_{\overline{B}\in \A_p(A(N/H))} \A_p(BH) \join \A_p(A(N/H))_{>\overline{B}}
\end{equation}
By minimality, $\A_p(A(N/H))$ is simply connected. Now we show that $X_B=\A_p(BH) \join \A_p(A(N/H))_{>\overline{B}}$ is also simply connected for every $B\in \A_p(AN)$.

If $|B| = p^3$, then $X_B = \A_p(BH)$ and it is simply connected by induction.

If $|B| = p^2$, then $X_B$ is a join of a connected space (see Remark \ref{remarkConnected}) with a non-empty space. Thus, it is simply connected. Note that there are no maximal elements of order $p^2$ in $\A_p(AN)$ since $A$ is a Sylow $p$-subgroup of $AN$.

If $|B| = p$, then $\A_p(BH)$ is non-empty and $ \A_p(A(N/H))_{>\overline{B}}$ is connected by \cite[Theorem 2]{Asc93}. Therefore, $X_B$ is simply connected.

Since all the spaces in the wedge of Equation (\ref{equationWedge}) are simply connected, we deduce that $\A_p(AN)$ is simply connected, which is a contradiction. Therefore, $N$ has no $A$-invariant non-trivial proper normal subgroup. In particular, $N$ must be characteristicaly simple and therefore $N = L_1\times\ldots \times L_n$ is a direct product of isomorphic simple groups on which $A$ acts transitively. Consequently, $AN$ is in the hypotheses of Aschbacher's conjecture and $\A_p(AN)$ is simply connected.
\end{proof}

\begin{remark}
Note that in the proof of Lemma \ref{lemmaJoinApgLinks}, Aschbacher's conjecture only needs to be assumed on the $p'$-simple groups involved in $N$.
\end{remark}

Now we apply these results to reduce to the case $O_{p'}(G) = 1$. Assume $G = \Omega_1(G)$ and Aschbacher's conjecture for $p$-rank $3$. Then, by Lemmas \ref{lemmaPW} and \ref{lemmaJoinApgLinks} and Remark \ref{remarkConnected}, $$\pi_1(\A_p(G)) = \pi_1(\A_p(G/O_{p'}(G))) *_{\overline{B}\in \A_p(G/N)^2} \pi_1(\A_p(BN) * \A_p(G/N)^2_{>\overline{B}}).$$ All the groups $\pi_1(\A_p(BN) * \A_p(G/N)^2_{>\overline{B}})$ are free by Lemma \ref{lemmaJoinApgLinks}. Therefore, $\pi_1(\A_p(G))$ is free whenever $\pi_1(\A_p(G/O_{p'}(G)))$ is. Let $S_G = \Omega_1(G)/O_{p'}(\Omega_1(G))$.

\begin{corollary}\label{coroQuotientByOpprime}
Assume Aschbacher's conjecture for $p$-rank $3$. Then there is an isomorphism $\pi_1(\A_p(G)) \simeq \pi_1(\A_p(S_G)) * F$, where $F$ is a free group. In particular, $\pi_1(\A_p(G))$ is free if 
$\pi_1(\A_p(S_G))$ is free.
\end{corollary}

\begin{remark}
In Corollary \ref{coroQuotientByOpprime}, we only need Aschbacher's conjecture to hold on the $p'$-simple groups involved in $O_{p'}(\Omega_1(G))$.
\end{remark}

We finish this section with some remarks concerning Aschbacher's conjecture. Recall the statement of the conjecture.

\begin{aconjecture}
Let $G$ be a finite group such that $G = AF^*(G)$, where $A$ is an elementary abelian $p$-subgroup of rank $r\geq 3$ and $F^*(G)$ is the direct product of the $A$-conjugates of a simple component $L$ of $G$ of order prime to $p$. Then $\A_p(G)$ is simply connected.
\end{aconjecture}

Aschbacher showed that the conjecture holds for a wide class of simple groups $L$: the alternating groups, the groups of Lie type and Lie rank at least $2$, the Mathieu sporadic groups and the groups $L_2(q)$ with $q$ even (see \cite[Theorem 3]{Asc93}). The case of the Lyons sporadic group is proved in \cite{AS92}, and Segev dealt with many of the groups of Lie type and Lie rank $1$ in \cite{Seg94}.

In the following proposition we reduce the study of Aschbacher's conjecture to the $p$-rank $3$ case.

\begin{proposition}\label{propReductionAschbacherConjecture}
If Aschbacher's conjecture holds for $p$-rank $3$, then it holds for any $p$-rank $r\geq 3$. Moreover, if the conjecture holds in $p$-rank $3$ for a $p'$-simple group $L$ then it holds in any $p$-rank $r\geq 3$ for $L$.
\end{proposition}

\begin{proof}
We use again Pulkus-Welker's decomposition restricted to $\A_p(G)^2$, as in Lemma \ref{lemmaPW}.

Let $G=AF^*(G)$ be as in the statement of the conjecture, where $m_p(A)\geq 4$. Assume that the conjecture holds for $p$-rank $3$. Set $N=F^*(G)= L_1\times\ldots\times L_n$ where the $L_i\ 's$ are the components of $G$, all isomorphic to $L$. By Lemma \ref{lemmaPW}, 

\[\A_p(G)^2\simeq \A_p(G/N)^2 \bigvee_{\overline{B}\in \A_p(G/N)^2} \A_p(BN)*\A_p(G/N)^2_{>\overline{B}}\]

Since $G/N\simeq A$, we may replace $\A_p(G/N)$ by $\A_p(A)$ and take $\overline{B} = B\leq A$ without loss of generality. The poset $\A_p(A)$ is contractible, and in particular simply connected.

Note that $\A_p(A)_{>B}^2$ is isomorphic to $\A_p(A/B)^{2-m_p(B)}$. Therefore, it is $0$-connected (resp. $(-1)$-connected) for $|B|=p$ (resp. $p^2$).

On the other hand, $\A_p(BN)$ is $(-1)$-connected and $0$-connected for $|B|=p$ and $p^2$, respectively (see Remark \ref{remarkConnected}). Thus, $\A_p(BN) * \A_p(A)^2_{>B}$ is simply connected when $|B| = p$ or $p^2$.

It remains to show that $\A_p(BN)$ is simply connected if $|B|=p^3$. We may assume $C_B(N) = 1$, so $B$ acts faithfully on $N$. If $B$ acts transitively on the set $\{L_1,\ldots,L_n\}$, we are in the hypotheses of Aschbacher's conjecture for $p$-rank $3$ and $\A_p(BN)$ is simply connected. If the action of $B$ on $\{L_1,\ldots,L_n\}$ is not transitive, then we may write $N = N_1\times\ldots\times N_m$ where each $N_j$ is a $B$-invariant direct product of some of the $L_i\ 's$. Then, inductively and by hypothesis, $\A_p(B(N/N_1))$ and $\A_p(CN_1)$ are simply connected for each $C\leq BN$ of $p$-rank $3$. By Lemma \ref{lemmaPW}, we conclude that $\A_p(BN)$ is simply connected.
\end{proof}

\section{Reduction to the almost simple case}\label{sectionfive}

In this section, we reduce the study of freeness of the fundamental group to the almost simple case. The main result of this section is the following.

\begin{theorem}\label{mainTheorem}
Let $G$ be a finite group and $p$ a prime dividing $|G|$. Assume that Aschbacher's conjecture holds. Then there is an isomorphism $\pi_1(\A_p(G)) = \pi_1(\A_p(S_G)) * F$, where $F$ is a free group. Moreover, $\pi_1(\A_p(S_G))$ is a free group (and therefore $\pi_1(\A_p(G))$ is free) except possibly if $S_G$ is almost simple.
\end{theorem}

Note that by Corollary \ref{coroQuotientByOpprime}, we only need to prove the moreover part. If the theorem does not hold, we can take a minimal counterexample $G$, and then we can assume that $G$ satisfies the following conditions:

\begin{enumerate}
\item[(C1)] $G = \Omega_1(G)$ and $\A_p(G)$ is connected,
\item[(C2)] $O_p(G) = 1$ (since otherwise $\A_p(G)$ is contractible by Proposition \ref{propOpGContractible}),
\item[(C3)] $\pi_1(\A_p(G))$ is not a free group. In particular, $\A_p(G)$ is not simply connected and $m_p(G)\geq 3$,
\item[(C4)] $O_{p'}(G) = 1$ (by minimality and Corollary \ref{coroQuotientByOpprime}),
\item[(C5)] $G\not\simeq G_1\times G_2$ (by Proposition \ref{propQuillenJoin}).
\end{enumerate}

\begin{remark}\label{remarksobrecond}
From conditions (C2) and (C4) we deduce that  $Z(G) = 1$, $Z(E(G)) = 1$, $F(G) = 1$ and $F^*(G) = L_1\times \ldots\times L_r$ is the direct product of simple components of $G$, each one of order divisible by $p$. In particular $C_G(F^*(G)) = Z(E(G)) = 1$, so $F^*(G)\leq G\leq\Aut(F^*(G))$.
\end{remark}

\begin{remark}\label{remarksobreconditions}
If $G$ satisfies the above conditions, by Remark \ref{remarksobrecond}, $F^*(G) = L_1\times \ldots\times L_r$. Therefore  $\A_p(F^*(G))$ has free fundamental group if $r = 2$, and it is simply connected for $r > 2$ (see Proposition \ref{propQuillenJoin}). If $r = 1$, $G$ is almost simple. We deal with the cases $r=2$ and $r>2$ separately (see Theorems \ref{theorem2Components} and \ref{theoremPComponents} below).
\end{remark}

In what follows, we do not need to assume Aschbacher's conjecture. In \cite[Sections 7 \& 8]{Asc93}, Aschbacher characterized the groups $G$ for which some link $\A_p(G)_{>T}$, with $T\leq G$ of order $p$, is disconnected. The following proposition deals with the case of connected links. Concretely, \cite[Theorem 1]{Asc93} asserts that if $O_{p'}(G) = 1$ and the links $\A_p(G)_{>T}$ are connected for all $T\leq G$ of order $p$, then either $\A_p(G)$ is simply connected, $G$ is almost simple and $\A_p(G)$ and $\A_p(F^*(G))$ are not simply connected, or else $G$ has certain particular structure. We prove that in the later case, the fundamental group is free.

\begin{proposition}\label{connectedLinksCase}
Suppose $G$ satisfies conditions (C1)...(C5). If the links $\A_p(G)_{>T}$ are connected for all $T\leq G$ of order $p$, then $G$ is almost simple and $\A_p(F^*(G))$ is not simply connected.
\end{proposition}

\begin{proof}
We use \cite[Theorem 1]{Asc93}. By conditions (C1)...(C5), $G$ corresponds either to case (3) or case (4) of \cite[Theorem 1]{Asc93}. Case (4) implies that $G$ is almost simple and $\A_p(F^*(G))$ is not simply connected. If $G$ is in case (3) of \cite[Theorem 1]{Asc93}, then $\pi_1(\A_p(G))$ is free (which is a contradiction by (C3)). This is deduced from the proof of \cite[(10.3)]{Asc93}, since under these hypotheses $\A_p(G)$ and $\A_p(F^*(G))$ are homotopy equivalent, and $\pi_1(\A_p(F^*(G)))$ is free by Proposition \ref{propQuillenJoin}.
\end{proof}

\begin{remark}
In \cite[Theorem 1]{Asc93}, Aschbacher's conjecture is required. However, since we are assuming $O_{p'}(G) = 1$, we do not need to assume the conjecture in the above proposition.
\end{remark}


For the rest of this section we will assume that $G$ is not almost simple, so $F^*(G) = L_1\times \ldots \times  L_r$ with $r > 1$. We deal with the cases $r=2$ and $r>2$ separately.

\begin{remark}\label{remarkBenderGroups}
Let $L$ be a simple group with a strongly $2$-embedded subgroup, i.e. such that $\A_2(L)$ is disconnected. Then $L$ is a simple group of Lie type and Lie rank $1$ in characteristic $2$ and it is isomorphic to one of the simple groups $L_2(2^n), U_3(2^n)$ or $Sz(2^{2n+1})$ by Theorem \ref{disconnectedCasesTheorem}. In any case, the Sylow $2$-subgroups of $L$ have the trivial intersection property. That is, $P\cap P^g = 1$ if $g\in L - N_L(P)$ (see \cite[Theorem 7]{Sei82}). Therefore, $\A_2(L)$ has $|\Syl_2(L)|$ connected components. If $C$ is a connected component of $\A_2(G)$, $C  = \A_2(P)$ for some $P\in \Syl_2(L)$ and thus, $\A_2(P)$ is contractible by Proposition \ref{propOpGContractible}. In particular, the components of $\A_2(L)$ are simply connected.
\end{remark}

\begin{theorem}\label{theorem2Components}
Under conditions (C1)...(C5), if $F^*(G) = L_1\times L_2$ is a direct product of two simple groups, then $p = 2$, $G \simeq L\wr C_2$ (the standard wreath product), with $L$ a simple group of Lie type and Lie rank $1$ in characteristic $2$, $L_1\simeq L_2\simeq L$ and $\pi_1(\A_2(G))$ is a free group with $(|\Syl_2(L)|-1)(|\Syl_2(L)|-1 +|L|)$ generators.
\end{theorem}

\begin{proof}
Note that $\A_p(F^*(G))$ is homotopy equivalent to $\A_p(L_1)*\A_p(L_2)$, which is simply connected if and only if $\A_p(L_1)$ or $\A_p(L_2)$ is connected (see Proposition \ref{propQuillenJoin}).

Assume $\A_p(F^*(G))$ is simply connected. Since $\A_p(G)$ is not simply connected,  by Lemma \ref{lemmaConnectedLinks} there exists some subgroup $T\leq G$ of order $p$ such that $\A_p(C_{F^*(G)}(T))$ is disconnected. Since $F^*(G) = L_1\times L_2$, $m_p(F^*(G)) > 2$ by simple connectivity. By \cite[(10.5)]{Asc93} $T$ acts regularly on the set of components of $G$ and each $L_i$ has a strongly $p$-embedded subgroup. In particular $p = 2$ and $L_1\simeq L_2$. Since $\A_p(F^*(G)) \simeq  \A_p(L_1) * \A_p(L_2)$ is simply connected, then $\A_p(L_i)$ is connected for some $i$ by Proposition \ref{propQuillenJoin}, so $L_i$ does not have a strongly $p$-embedded subgroup, which is a contradiction.

Now suppose $\A_p(F^*(G))$ is not simply connected. Then, $\pi_1(\A_p(F^*(G)))$ is a free group by Proposition \ref{propQuillenJoin}, and $L_1$ and $L_2$ are simple groups with strongly $p$-embedded subgroups. We use \cite[(10.3)]{Asc93}. By the above hypotheses, $G$ corresponds to either case (2) or case (3) of \cite[(10.3)]{Asc93}. In case (3), as we mentioned in the proof of Proposition \ref{connectedLinksCase}, $\A_p(G)$ and $\A_p(F^*(G))$ are homotopy equivalent (which contradicts the conditions on $G$). Therefore, $G$ is in case (2) of \cite[(10.3)]{Asc93}, $p=2$ and $G = L_1 \wr T$ for some $T\leq G$ of order $2$. Then, $T = \gen{t}$ for an involution $t \in G$, and $L_1$ is a group of Lie type and Lie rank $1$ in characteristic $2$ by Remark \ref{remarkBenderGroups}.

We prove now that $\pi_1(\A_2(G))$ is free.  Let $N = F^*(G) = L_1\times L_2$ and $L = L_1$. Let $\N = \{A\in \A_2(G) : A\cap N \neq 1\}$ and let $\S = \A_p(G) - \N$ be the complement of $\N$ in $\A_2(G)$. If $A\in\S$ then $A\simeq AN/N\leq TN/N = T$, then $|A| = 2$. Therefore, $\S = \{A\leq G: |A| = 2, A\nleq N\}$ consists of some minimal elements of the poset, and we have
$$\A_2(G)=\N\bigcup_{A\in \S} \A_2(G)_{\geq A}.$$

 For each $A\in \S$, $\N\cap  \A_2(G)_{\geq A} = \{W\in \A_2(G): W\cap N\neq 1, W\geq A\}\simeq \A_2(C_N(A))$ by Lemma \ref{lemmaLinkRetract}. Note that $C_N(A) \simeq L$ since $G = NA$. By Remark \ref{remarkBenderGroups}, $\A_2(L)$ has $|\Syl_2(L)|$ connected components and each component is simply connected. Since  $\A_2(G)_{\geq A}$ is contractible, then, by the non-connected version of van Kampen theorem (see \cite[Section 9.1]{Bro}), $\pi_1(\N\cup \A_2(G)_{\geq A}) = \pi_1(\N)* F_A$, where $F_A$ is the free group of rank $|\Syl_2(L)| - 1$. Since $\A_2(G)_{\geq A}\cap \A_2(G)_{\geq B}\subseteq \N$ for each 
 $A\neq B\in \S$, then recursively we have $\pi_1(\A_2(G))\simeq \pi_1(\N)*F$, where $F$ is the free group of rank $(|\Syl_2(L)|-1)|\S|$.

By Remark \ref{remarkInflationSubsapce}, $\N \simeq \A_2(N) = \A_2(L_1\times L_2)$. Therefore $\pi_1(\N)$ is a free group on $(|\Syl_2(L)|-1)^2=(|\pi_0(\A_2(L_1))|-1)(|\pi_0(\A_2(L_2))|-1)$ generators by Proposition \ref{propFreePi1Join}.

Finally we compute $|\S|$. Let $\mathcal{I}(G)$ be the number of distinct involutions in $G$. Therefore, $\mathcal{I}(G) = \mathcal{I}(N) + s$, where $s$ is the number of involutions not contained in $N$. Note that $s = |\S|$. If $g\in G - N$ is an involution, then $g = xyt$ with $x\in L_1$ and $y\in L_2$. The condition $g^2=1$ implies $1 = xytxyt = xy x^t y^t = (x y^t)(yx^t)$ with $y^t\in L_1$ and $x^t\in L_2$. Since $L_1\cap L_2 = 1$, $xy^t = 1$ and $yx^t = 1$, i.e. $y = (x^{-1})^t$. Therefore, $g = x (x^{-1})^t t = x t x^{-1}$ and $s = |\{ x(x^{-1})^t t : x\in L_1\}| = |L_1| = |L|$.

In conclusion, $ \pi_1(\A_2(G))$ is a free group with $(|\Syl_2(L)|-1)^2 + |L|(|\Syl_2(L)|-1)$ generators.
\end{proof}

Now we deal with the case $r > 2$.

\begin{theorem}\label{theoremPComponents}
Under conditions (C1)...(C5), if $F^*(G) = L_1\times \ldots \times L_r$ is a direct product of simple groups with $r > 2$, then $p$ is odd, $r = p$, each $L_i$ has a strongly $p$-embedded subgroup, $\{L_1,\ldots,L_r\}$ is permuted regularly by some subgroup of order $p$ of $G$, and $\pi_1(\A_p(G))$ is a free group.
\end{theorem}

\begin{proof}
The hypotheses imply that $\A_p(F^*(G)) \simeq \A_p(L_1) * \ldots * \A_p(L_r)$ is simply connected by Proposition \ref{propQuillenJoin}. Then there exists some subgroup $T\leq G$ of order $p$ such that $\A_p(C_{F^*(G)}(T))$ is disconnected by Lemma \ref{lemmaConnectedLinks}. We apply \cite[(10.5)]{Asc93}. The hypotheses imply that we are in case (5) of \cite[(10.5)]{Asc93}. Then $T$ permutes regularly the components $\{L_1,\ldots, L_r\}$ and each $L_i$ has a strongly $p$-embedded subgroup. In particular, $r = p$ is odd and $L_i\simeq L_j$ for all $i,j$.

Set $N = F^*(G)$ and let $H = \bigcap_i N_G(L_i)$. Then $N\normal H$. If $A\in \A_p(H)$, then $A\leq \bigcap_i N_G(L_i)$, so $C_N(A) = \prod_i C_{L_i}(A)$. In particular, $$\A_p(C_N(A))  = \A_p \left(\prod_i C_{L_i}(A)\right) \simeq \A_p(C_{L_1}(A))* \A_p(C_{L_2}(A)) * \ldots * \A_p(C_{L_p}(A))$$ is simply connected. Therefore, by Lemma \ref{lemmaConnectedLinks}, $\A_p(H)$ is simply connected, hence $\mathcal{N} = \{X\in \A_p(G) : X\cap H \neq 1\}$ is also simply connected by Remark \ref{remarkInflationSubsapce}.  Consider the complement $\S = \A_p(G) - \mathcal{N}$. If $X\in \S$ then $X\cap H = 1$. Thus, $X = X_1X_2$ where $X_2$ permutes regularly the components $\{L_1,\ldots, L_p\}$ and $X_1\leq \bigcap_i N_G(L_i) = H$. Since $X\cap H = 1$, we conclude that $X_1 = 1$ and $|X_2| = p$, i.e. $|X| = p$. Therefore, $\S$ is an anti-chain and, by Proposition \ref{propCollapsingSCSubspace},   $\pi_1(\A_p(G))$ is free.
\end{proof}

\begin{proof}[Proof of Theorem \ref{mainTheorem}]
We may assume that $G = \Omega_1(G)$, $\A_p(G)$ is connected and $O_p(G) = 1$. By Corollary \ref{coroQuotientByOpprime}, we only need to prove that $\pi_1(\A_p(G))$ is free when $O_{p'}(G) = 1$ and $G$ is not almost simple. We may assume conditions (C1)...(C5), and the proof now follows from Remark \ref{remarksobreconditions} and Theorems \ref{theorem2Components} and \ref{theoremPComponents}.
\end{proof}

\section{Freeness in some almost simple cases}\label{sectionsix}

In this section we prove that $\pi_1(\A_p(G))$ is free when $G$ is an almost simple group with some extra hypothesis. We will use the structure of the outer automorphism group of a simple group. We refer the reader to sections 7 and 9 of \cite{GL83} and Chapters 2 to 5 of \cite{GLS98}. For the $p$-rank of simple groups we will use the results of section 10 of \cite{GL83} and in particular \cite[(10-6)]{GL83}.

Consider a finite group $G$ such that $L\leq G\leq \Aut(L)$, where $L$ is a simple group of order divisible by $p$. We may suppose that $G = \Omega_1(G)$, $m_p(G) \geq 3$ and $\A_p(G)$ is connected.

\begin{theorem}\label{theoremDiscOrSC}
Let $G$ and $L$ be as above. Then $\pi_1(\A_p(G))$ is free if $\A_p(L)$ is disconnected or simply connected.
\end{theorem}

\begin{proof}
We prove first that $\pi_1(\A_p(G))$ is free when $\A_p(L)$ is disconnected. In this case, $L$ has a strongly $p$-embedded subgroup. We deal with each case of the list of Theorem \ref{disconnectedCasesTheorem}.

\begin{itemize}
\item If $m_p(L) = 1$, then $p$ is odd and $m_p(G) \leq 2$ by \cite[(7-13)]{GL83}.
\item If $L$ is a simple group of Lie type and Lie rank $1$ in characteristic $p$, then the Sylow $p$-subgroups of $L$ have the trivial intersection property, i.e. $P\cap P^g = 1$ if $P\in \Syl_p(L)$ and $g\in L-N_L(P)$ (see \cite[Theorem 7]{Sei82}). The proof is similar to the proofs of Theorems \ref{theorem2Components} and \ref{theoremPComponents}. Let $\N = \{X\in \A_p(G) : X\cap L \neq 1\}$ and let $\S = \A_p(G) - \N$. Since $m_p(\Out(L)) \leq 1$, $\S$ consists of subgroups of order $p$. By Remarks  \ref{remarkInflationSubsapce} and \ref{remarkBenderGroups}, $\N\simeq \A_p(L)$ has simply connected components. If $A\in \S$, then $\A_p(G)_{\geq A} \cap \N\simeq \A_p(C_L(A))$ by Lemma \ref{lemmaLinkRetract}, and the Sylow $p$-subgroups of $C_L(A)$ intersect trivially, so 
 $\A_p(G)_{\geq A} \cap \N$ has simply connected components. Then $\pi_1(\A_p(G))$ is free by van Kampen theorem.
\item $L\not\simeq {}^2G_2(3)$ or $\Aut(Sz(32))$ since these groups are not simple.
\item In the remaining cases, $m_p(G) = 2$ by \cite[(10-6)]{GL83} or by direct computation.
\end{itemize}

Now we prove that $\pi_1(\A_p(G))$ is free when $\A_p(L)$ is simply connected. Note that $m_p(L)\geq 3$ since otherwise $O_p(L)\neq 1$, contradicting that $L$ is simple. By Lemma \ref{lemmaConnectedLinks}, we may assume that $\A_p(C_L(T))$ is disconnected for some $T\leq G$ of order $p$. Therefore, we are dealing with one of the cases (1), (2), (3) or (4) of \cite[(10.5)]{Asc93}. We deal with each one of them.
\begin{enumerate}
\item If $L$ is of Lie type and Lie rank $1$ in characteristic $p$, then $\A_p(L)$ is disconnected, contradicting the hypothesis.
\item If $p = 2$, $q$ is even and $L\simeq L_3(q)$, $U_3(q)$ or $Sp_4(q)$, then $L$ is of Lie type and Lie rank at most $2$. In any case, $\A_p(L)$ is not simply connected since it has the homotopy type of a wedge of spheres of dimension equal to the Lie rank of $L$ minus $1$ (see \cite[Theorem 3.1]{Qui78}). If $L \simeq G_2(3)$, then $\Out(G_2(3)) = C_2$. Hence, $\pi_1(\A_p(G))$ is free by Proposition \ref{propCollapsingSCSubspace} applied to $Y = \{X\in \A_p(G):X\cap L\neq 1\}\subseteq \A_p(G)$. Note that $Y$ is simply connected since $Y\simeq \A_p(L)$ by Remark \ref{remarkInflationSubsapce}.
\item If $p = 2$ and $L\simeq L_3(q^2)$ with $q$ even, then $L$ has Lie rank $2$, and thus $\A_p(L)$ is a wedge of $1$-spheres, contradicting the hypothesis (see \cite[Theorem 3.1]{Qui78}).
\item If $p > 3$, $q\equiv \epsilon \mod p$ and $L\simeq L^{\epsilon}_p(q)$, then $m_p(\Out(L)) = 1$ by \cite[(9-3)]{GL83}. Therefore, $\pi_1(\A_p(G))$ is free by Proposition \ref{propCollapsingSCSubspace} applied to $Y = \{X\in \A_p(G):X\cap L\neq 1\}\subseteq \A_p(G)$.
\end{enumerate}
\end{proof}

\begin{corollary}\label{corollarylietype}
If $L$ is a Lie type group in characteristic $p$ and $p\nmid (G:L)$ when $L$ has Lie rank $2$, then $\pi_1(\A_p(G))$ is free.
\end{corollary}

\begin{proof}
Since $\A_p(L)$ is homotopy equivalent to the building associated to $L$ (see \cite[Theorem 3.1]{Qui78}), it is a bouquet of spheres of dimension $n-1$, where $n$ is the Lie rank of $L$. If $n \neq 2$, $L$ is in the hypotheses of Theorem \ref{theoremDiscOrSC}. If $n = 2$, $\A_p(G) = \A_p(L)$ since $p\nmid (G:L)$. In either case, $\pi_1(\A_p(G))$ is a free group.
\end{proof}

\begin{remark}
Corollary \ref{corollarylietype} does not give information in the case that $L$ has Lie rank $2$ and  $p\mid (G: L)$. By  \cite[Theorem 3.1]{Qui78}, $\pi_1(\A_p(L))$ is a free group but if $p\mid (G:L)$ then it may happen that $\pi_1(\A_p(G)) \not\simeq \pi_1(\A_p(L))$. For example, take $L = L_3(4)$ and $p = 2$. Note that $\Out(L) = D_{12}$.  We have computed $\pi_1(\A_p(G))$ for all possible groups $G$, with $L \leq G\leq \Aut(L)$, and they turned out to be free. If $G = \Aut(L)$, $\A_p(G)$ is simply connected, while $\pi_1(\A_p(L))$ is a non-trivial free group.
\end{remark}

Recall the classification of J. Walter of simple groups with abelian Sylow $2$-subgroup \cite{Wal69}.

\begin{theorem}[{\cite{Wal69}}]\label{theoremwalter}
Let $L$ be a a simple group with abelian Sylow $2$-subgroup $P$. Then $L$ is isomorphic to one of the following groups:
\begin{enumerate}
\item $L_2(q)$, $q\equiv 3,5\mod 8$ (and $P$ is elementary abelian of order $2^2$),
\item $L_2(2^n)$, $n\geq 2$ (and $P$ is elementary abelian of order $2^n$),
\item ${}^2G_2(3^n)$, $n$ odd (and $P$ is elementary abelian of order $2^3$),
\item $J_1$ (and $P$ is elementary abelian of order $2^3$).
\end{enumerate}
\end{theorem}

In the proof of the next result, we will work with Bouc poset $\B_p(G)$ of non-trivial $p$-radical subgroups instead of $\A_p(G)$ (see Section \ref{sectiontwo}).

\begin{theorem}
Suppose $G$ is almost simple, $p = 2$ and $F^*(G)$ has abelian Sylow $2$-subgroups. Then $\pi_1(\A_p(G))$ is free.
\end{theorem}

\begin{proof}
Let $L = F^* (G)$. By Walter's Theorem \ref{theoremwalter}, $L$ is one of the groups (1)...(4). 

The case (2) follows from the disconnected case of Theorem \ref{theoremDiscOrSC}. 

In case (4),  $G = J_1$ and $\pi_1(\A_2(G))$ is free on $4808$ generators by computer calculation.

In case (3), $L = {}^2G_2(3^n)$, with $n$ odd. Then $\Out(L) = C_n$ has odd order and $\A_2(G) = \A_2(L)$. It suffices to prove that $\K(\B_p(L))$ has dimension $1$. Note that $\S_2(L) = \A_2(L)$. Let $q=3^n$. By \cite[Theorem 6.5.5]{GLS98}, the normalizers of the non-trivial $2$-subgroups have the following forms: $C_2\times L_2(q)$ for involutions and $(C_2^2\times D_{(q+1)/2})\rtimes C_3$ for four-subgroups. If $t\in L$ is an involution, $O_2(N_L(t)) = \gen{t}$. If $X$ is a four-subgroup of $L$ then $X < O_2(N_L(X)) \simeq C_2^3$ since $q\equiv 3 \mod 4$. For a Sylow $2$-subgroup $S$ of $L$ we have that $S = O_2(N_L(S))$. Therefore, the poset $\B_2(L)$ contains the subgroups generated by one involution and the Sylow $2$-subgroups. In consequence, $\K(\B_2(L))$ is a $1$-dimensional simplicial complex (and therefore it has free fundamental group).

In case (1), $L= L_2(q)$, $q\equiv 3,5\mod 8$ then $q$ is odd and it is not a square. Therefore, $\Aut(L) / \Inndiag(L)$ has odd order and thus we may assume $L \leq G \leq \Inndiag(L)$. In any case, $m_2(G) = 2$ by \cite[Theorem 4.10.5(b)]{GLS98}.
\end{proof}

Now we compute the fundamental group of $\A_p(G)$ for some particular sporadic groups $L$. Note that $m_p(L)\leq 2$ if $p > 7$.

\begin{example}
By computer calculations, $\pi_1(\A_2(G))$ is free for $L = J_1$ or $J_2$. Note that $\Out(J_1) = 1$ and $\Out(J_2) = C_2$. Note that, if $p$ is odd, $m_p(G)\leq 2$ for $L = J_1$ or $J_2$ (see \cite[(10-6)]{GL83}).
\end{example}

\begin{example}
If $G = J_3$ or $O' N$ and $p = 3$, then $\pi_1(\A_p(G))$ is free. By \cite[Proposition 3.1.4]{Kot97} and \cite[Section 6.1]{UY02}, there are only two conjugacy classes of non-trivial $p$-radical subgroups of $G$. Therefore, $\K(\B_p(G))$ has dimension $1$. For $p > 3$, $m_p(G) \leq 2$, so $\pi_1(\A_p(G))$ is free.
\end{example}

\begin{example}
If $G = Mc$ and $p = 3$, then $\pi_1(\A_p(G))$ is free. By computer calculations, if $S$ is a Sylow $3$-subgroup of $G$, there exist three subgroups of $S$ (up to conjugacy) which are $p$-radical subgroups of $G$: $A$, $B$ and $S$. Their orders are $|A| = 81$, $|B| = 243$, $|S| = 729$. Moreover, $A,B\normal S$ and $A\nleq B$. Then $A^g\nleq B$ for any $g\in G$ such that $A^g\leq S$, and therefore $\K(\B_p(G))$ is $1$-dimensional. For $p > 3$, $m_p(G) \leq 2$, so $\pi_1(\A_p(G))$ is free.
\end{example}

\begin{proposition}
Assume $L$ is a Mathieu sporadic group. If $p$ is odd and $m_p(G) \geq 2$, then $\A_p(G)$ has free fundamental group. If $p = 2$, $\A_2(G)$ is simply connected except for $L = M_{11}$, in which case $\pi_1(\A_2(G))$ is a non-trivial free group.
\end{proposition}

\begin{proof}
Let $L$ be a one of the Mathieu groups $M_{11}$, $M_{12}$, $M_{22}$, $M_{23}$ or $M_{24}$. In all cases, $m_p(L)\leq 2$ if $p$ is odd. Therefore, we may assume that $p = 2$. Recall that $\Aut(L) = L$ for $L = M_{11}$, $M_{23}$ and $M_{24}$, and $\Out(L) = C_2$ for $L = M_{12}$ and $M_{22}$. 

Note that $m_2(M_{11}) = 2$. For $L = M_{12}$ or $M_{22}$ we checked with GAP that $\A_2(L)$ is simply connected.

If $G = \Aut(M_{22})$, then $\S=\{X\in \A_p(G) : X \cap M_{22} = 1\}\subseteq \Min(\A_p(G))$. Let $\N = \A_p(G) - \S$. Recall that $\N\simeq \A_2(M_{22})$ by Remark \ref{remarkInflationSubsapce} and therefore it is simply connected. Any $A\in \S$ is generated by an involution acting by outer automorphism on $M_{22}$. By \cite[Table 5.3c]{GLS98}, its centralizer in $M_{22}$ has a non-trivial normal $2$-subgroup. That is, $\A_2(C_{M_{22}}(A))$ is contractible by Proposition \ref{propOpGContractible}. Then for any $A\in\S$, $\N \cup \A_2(G)_{\geq A}$ is simply connected by van Kampen theorem and, recursively, $\A_2(G)$ is simply connected. A similar reasoning shows that $\A_2(G)$ is simply connected for $G = \Aut(M_{12})$ (see \cite[Table 5.3b]{GLS98}).

By \cite[p.295]{Smi11}  $\A_2(M_{24})$ is homotopy equivalent to its $2$-local geometry, which is simply connected. 

It remains to determine the fundamental group of $\A_2(M_{23})$. For this we use the classification of the maximal subgroups of $M_{23}$ and $M_{22}$. First note that $M_{22}$ is a maximal subgroup of $M_{23}$ of odd index $23$. In particular, any elementary abelian $2$-subgroup of $M_{23}$ is contained in some conjugate of $M_{22}$. Therefore, $\U=\{ \A_2(M_{22}) \cup \A_2(M_{22}^g) : g\in M_{23}\}$ is a cover of $\A_2(M_{23})$ by subcomplexes. We have computed the intersections between different conjugates of $M_{22}$ with GAP. All the intersections $M_{22}\cap M_{22}^g$, with $g\in M_{23} - M_{22}$, form a subgroup of $M_{22}$ of order $20160$. All the maximal subgroups of $M_{22}$ have order less than $20160$ except for the maximal subgroup $L_3(4)$ (and all its conjugates) which have order exactly $20160$. Thus, $M_{22}\cap M_{22}^g = L_3(4)$ and, by van Kampen theorem, each element of $\U$ is simply connected. The triple intersections of different conjugates of $M_{22}$ are all isomorphic to $C_2^2\rtimes A_8$, and the quadruple intersections of different conjugates of $M_{22}$ are all isomorphic to $C_2^4\rtimes C_3$. This shows that double and triple intersections of elements of $\U$ are connected. In consequence, by van Kampen theorem, $\A_2(M_{23})$ is simply connected.
\end{proof}

We investigate now the fundamental group of the Quillen complex of alternating groups at $p=2$. We use Ksontini's results on $\pi_1(\A_2(S_n))$ (see Section \ref{sectiontwo}).

\begin{remark}
By Theorem \ref{theoremDiscOrSC} the poset $\A_2(A_n)$ is disconnected if and only if $n = 5$,  since, for $p = 2$, the unique isomorphism of an alternating group with a group of the list is $A_5 \simeq L_2(4)$ (see \cite[(3-3)]{GL83}).
\end{remark}

\begin{proposition}\label{propAnCase}
Let $n\geq 4$. The fundamental group of $\A_2(A_n)$ is simply connected for $n=4$ and $n\geq 8$. For $n=5$, each component of $\A_2(A_5)$ is simply connected, for $n=6$ it is free of rank $16$, and for $n=7$ it is free of rank $176$. 
\end{proposition}

\begin{proof}
If $n = 4$, then $O_2(A_4)\neq 1$ and $\A_2(A_4)$ is contractible. If $n = 5$, $A_5 = L_2(4)$, which has trivial intersections of Sylow $2$-subgroups by \cite[Theorem 7]{Sei82}. Therefore, its connected components are simply connected. The cases $n = 6, 7, 8$ can be obtained directly by computer calculations.

We prove the case $n\geq 9$. We proceed similarly as before. Take $\N = \{A\in \A_2(S_n) : A\cap A_n\neq 1\}$. By Remark \ref{remarkInflationSubsapce}, $\N\simeq \A_2(A_n)$, and let $\S = \A_2(S_n) - \N$ be its complement in $\A_2(S_n)$. Note that $\S$ consists of the subgroups of order $2$ of $S_n$ generated by involutions which can be written with an odd number of disjoint transpositions. Note that for any $A\neq B\in\S$, $\A_2(S_n)_{\geq A} \cap \A_2(S_n)_{\geq B}\subseteq\N$. By Ksontini's Theorem \ref{theoremKsontini}, 
 $\pi_1(\A_2(S_n))=1$. Therefore, by van Kampen theorem, in order to prove that $\pi_1(\A_2(A_n))=\pi_1(\N)$ is trivial, we only need to show that the intersections
 $\N\cap \A_2(S_n)_{\geq A} \simeq \A_2(C_{A_n}(A))$ are simply connected for all $A\in\S$.
 
  We appeal now to the characterization of the centralizers of involutions in $A_n$ to show that $\A_2(C_{A_n}(x))$ is simply connected if $\gen{x}\in \S$. Let $x\in S_n - A_n$ be an involution acting as the product of $r$ disjoint transpositions and with $s$ fixed points.  By \cite[Proposition 5.2.8]{GLS98}, $C_{A_n}(x) \simeq (H_1\times H_2)\gen{t}$, where $H_1 \leq \ZZ_2\wr S_r$ has index $2$, and $H_2 \simeq A_s$. Here, the wreath product is taken with respect to the natural permutation of $S_r$ on the set $\{1,\ldots,r\}$. Moreover, $H_1 = E \rtimes S_r$ where $E\leq \ZZ_2^r$ is the subgroup $\{(a_1,\ldots, a_r)\in \ZZ_2^r : \sum_i a_i = 0\}$. If $s\leq 1$ then $t = 1$. If $s\geq 2$, then $t\neq 1$, $H_1\gen{t} = \ZZ_2\wr S_r$ and $H_2\gen{t} = S_s$. In any case, $E\normal (H_1\times H_2)\gen{t}$ since $[H_1,H_2] = 1$ and $E\normal H_1\gen{t}$. Therefore, if $r > 1$, $O_2(C_{A_n}(x))\neq 1$ and $\A_2(C_{A_n}(x))$ is contractible. In particular it is simply connected. If $r = 1$, $C_{A_n}(x) \simeq H_2\gen{t}\simeq S_s$, and therefore $\A_2(C_{A_n}(x))\simeq \A_2(S_s)$ is simply connected by Theorem \ref{theoremKsontini} (note that $s\geq 7$ since $2r + s = n\geq 9$).
\end{proof}

Combining Proposition \ref{propAnCase} with Theorem \ref{theoremKsontini} we deduce the following corollary.

\begin{corollary}
If $A_n\leq G\leq \Aut(A_n)$, then $\pi_1(\A_2(G))$ is a free group.
\end{corollary}

\begin{proof}
If $n\neq 6$, then $G = A_n$ or $S_n$. In any case, $\pi_1(\A_2(G))$ is free by Proposition \ref{propAnCase} and Theorem \ref{theoremKsontini}. For $n = 6$, $\Out(A_6) = C_2\times C_2$ and $A_6 < S_6 < \Aut(A_6)$. If $F^*(G) = A_6$, then $G\simeq S_6, A_6$ or $\Aut(A_6)$. In either case, $\pi_1(\A_2(G))$ is free by the above results or by computer calculations.
\end{proof}

\section*{Appendix}

In order to compute fundamental groups of Quillen complexes, we used GAP (see \cite{Gap}), together with the GAP package HAP (\cite{Hap}). We wrote two functions that compute the order complex of the Bouc poset $\B_p(G)$ of a given finite group $G$ at a given prime $p$, and its fundamental group.

For example, to compute the fundamental group of $\A_3(A_{10})$ we execute the following code in GAP.

\begin{lstlisting}[language=GAP]
gap> G:=AlternatingGroup(10);;
gap> BpG:=allRadicalSubgroups(G,3);;
gap> pi1BpG:=pi1(BpG,G,3);;
gap> pi1BpG;
\end{lstlisting}

Here below we exhibit the codes of these functions.

\begin{lstlisting}[language=GAP]
allRadicalSubgroups:=function(G,p)
	local BpS, BpG, H, N, a, b, i, g, add, numCC,
	normalizers, reducedList, subgroups, transversals;
	
	# We calculate first the non-trivial p-radical subgroups
	# contained in a fixed Sylow p-subgroup.
	
	S:=SylowSubgroup(G,p);
	subgroups:=Set(Filtered(SubgroupsSolvableGroup(S), l-> Order(l) > 1));
	BpS:=Filtered(subgroups, l-> Size(l) = Size(PCore(Normalizer(G,l),p)));
	
	# The list reducedList will contain exactly one representative for
	# each conjugacy class of each element in BpS.
	
	reducedList:=[];
	for a in BpS do
		add:=true;
		i:=1;
		while add and i <= Size(reducedList) do
			b:=reducedList[i];
			if IsConjugate(G,a,b) then
				add:=false;
			fi;
			i:=i+1;
		od;
		if add then
			Add(reducedList,a);
		fi;
	od;
	
	# numCC is the number of conjugacy classes of non-trivial
	# p-radical subgroups.
	
	numCC:=Size(reducedList);
	
	# For each representative of non-trivial p-radical subgroups,
	# we calculate its normalizer and a transversal of this.
	
	normalizers:=List(reducedList, H-> Normalizer(G,H));
	transversals:=List(normalizers, N->
	List(RightTransversal(G,N),i->CanonicalRightCosetElement(N,i)));
	
	# Now we compute Bp(G).
	
	BpG:=[];
	
	for i in [1..numCC] do
		for g in transversals[i] do
			H:=reducedList[i]^g;
			Add(BpG, H);
		od;
	od;
	
	return BpG;
end;;

allMaximalChains:=function(G, S, BpS, elementList, orders)

	# This function returns all maximal chains of non-trivial
	# p-radical subgroups of G.

	local g, i, j, o, t, x, H, Q, N, R, T, lastPoint, partialChain, 
	properSubgroups, radSubgroupsConj;
		
	N:=Normalizer(G,S);
	T:=List(RightTransversal(G,N), x-> CanonicalRightCosetElement(N,x));
	R:=[];
	o:=PositionSorted(orders, Size(S));
	
	for g in T do		
		Q:=[[ [o,PositionSorted(elementList[o],S^g)] ]];
		t:=1;
		radSubgroupsConj:=List(BpS, x-> x^g);
		while t<= Size(Q) do
			partialChain:=Q[t];
			lastPoint:=partialChain[Size(partialChain)];
			i:=lastPoint[1];
			j:=lastPoint[2];
			H:=elementList[i][j];
			
			properSubgroups:=Filtered(radSubgroupsConj, x-> 
			Size(x) < Size(H) and IsSubset(H, GeneratorsOfGroup(x)));
			
			properSubgroups:=List(properSubgroups, x->
			[PositionSorted(orders, Size(x)),
			PositionSorted(elementList[PositionSorted(orders,Size(x))],x)]);
			
			for x in properSubgroups do
				Add(Q, Concatenation(partialChain,[x]));
			od;
			t:=t+1;
		od;
		R:=Concatenation(R,Q);
	od;
	return R;
end;;

pi1:=function(BpG,G,p)
	local i, o, x, R, S, BpS, elementList, orders;
	
	S:=SylowSubgroup(G,p);
	BpS:=Filtered(BpG, x-> IsSubset(S, GeneratorsOfGroup(x)));
	
	orders:=Unique(List(BpS, Size));
	Sort(orders);
	
	elementList:=List(orders, x-> []);
	
	for o in orders do
		i:=PositionSorted(orders, o);
		elementList[i]:=Set(Filtered(BpG, x->Size(x) = o));
	od;
	
	# We calculate all maximal chains of BpG.
	
	R:=allMaximalChains(G, S, BpS, elementList, orders);
	
	# Finally, we compute the fundamental group of the
	# simplicial complex generated by these maximal chains.
	
	return FundamentalGroup(SimplicialComplex(R));
end;;
\end{lstlisting}

\end{document}